\documentclass[11pt]{article}
\usepackage{amssymb}
\usepackage{graphicx}

\newcommand{\Sk}{{\mathcal S}_k}
\newcommand{\sk}{{s}_k}
\newcommand{\skm}{{s}_{k,m}}
\newcommand{\binom}[2]{{{#1}\choose{#2}}}

\newtheorem{thm}{Theorem}[section]
\newtheorem{lemma}[thm]{Lemma}

\newenvironment{proof}{\paragraph{\textbf{Proof:}}}{\hfill$\square$}

\begin{document}

\title{Algorithms and Bounds on the Sums of Powers of Consecutive Primes}

\author{Cathal O'Sullivan \\
    Mathematics, Statistics, and Actuarial Science Department \\
    Computer Science and Software Engineering Department \\
    Butler University, Indianapolis IN, USA \\
    \texttt{cbosulli@butler.edu}
  \and Jonathan P. Sorenson \\
    Computer Science and Software Engineering Department \\
    Butler University, Indianapolis IN, USA \\
    \texttt{jsorenso@butler.edu}
  \and Aryn Stahl \\
    Physics and Astronomy Department \\
    Computer Science and Software Engineering Department \\
    Butler University, Indianapolis IN, USA \\
    \texttt{anstahl@butler.edu}
    }


\date{\today}
\maketitle

\begin{abstract}
  We present and analyze algorithms to count and enumerate
  all integers $n\le x$ that can be written as the sum of 
  consecutive $k$th powers of primes, for fixed $k>1$.

  We show that the number of such integers $n$ is asymptotically bounded
  by 
  $$
    c_k \frac{ x^{2/(k+1)} }{ (\log x)^{2k/(k+1)} },
  $$
  where $c_k$ is a constant depending solely on $k$, roughly $k^2$
  in magnitude.
  We also give an asymptotic lower bound of the form
  $$
	\frac{(k+1)^2}{2} \frac{ x^{2/(k+1)} }{ (\log x)^{2k/(k+1)} }.
  $$

  Our enumeration algorithm is optimal, in that it constructs all
    such integers $n\le x$ using a constant number of arithmetic
    operations for each $n$.
  Our algorithm to count such $n$ takes time proportional to
    the number of primes up to $x^{1/k}$, so that finding this list
    of primes is the bottleneck in its runtime.

  We also present some computational results, using our new algorithms,
    that provide empirical data to support our theoretical results.

  Our work extends the previous work by Tongsomporn, Wananiyakul, and Steuding
    \cite{TWS2022} who examined sums of squares of consecutive primes.
\end{abstract}

\section{Introduction\label{sec:intro}}

\nocite{BS,OEIS}

Let $\Sk(x)$ denote the set of integers $n\le x$ that can be written
  as a sum of the $k$th powers of consecutive primes.
For example, $5^3+7^3+11^3=1799$ is an element of ${\mathcal S}_3(2000)$.
Let $\sk(x)$ be the number of such $n$, counted with multiplicity.
If a specific integer $n$ has more than one representation as the sum
  of $k$th powers of consecutive primes,
  we count all such representations when we say "with multiplicity".
So we have $\sk(x)\ge \#\Sk(x)$.

In this paper, 
\begin{itemize}
\item We describe an algorithm that, given $k$ and $x$, 
  produces the elements of $\Sk(x)$ along with their representation.
		Its running time is linear in $\sk(x)$,
		the number of such representations. 
  The algorithm uses $O( kx^{1/k} )$ space.
  This is \S\ref{sec:alg}.
\item In \S\ref{sec:cnt},
  we describe a second algorithm that computes the value of $\sk(x)$,
  with multiplicity,
  given $k$ and $x$.
  This algorithm takes $O( x^{1/k}/\log\log x )$ arithmetic operations,
    the time it takes to find all primes up to $x^{1/k}$.
\item In \S\ref{sec:anal} we show that 
  $$
  s_k(x)\le (1+o(1))\cdot c_k \frac{ x^{2/(k+1)} }{ (\log x)^{2k/(k+1)} },
  $$
  where $c_k=(k^2/(k-1))\cdot(k+1)^{1-1/k}.$
  This is a generalization of a bound for $s_2(x)$ proven by \cite{TWS2022}.
  Their bound is explicit and ours is not.
  This is also an upper bound on the number of arithmetic operations
    used by our enumeration algorithm.
\item
  Also in \S\ref{sec:anal}, we give the lower bound
  $$
    s_k(x)\ge (1+o(1))\cdot \frac{(k+1)^2}{2} 
      \frac{ x^{2/(k+1)} }{ (\log x)^{2k/(k+1)} },
  $$
\item In \S\ref{sec:data} we apply our new algorithm to 
  compute $S_k(x)$ for various
  $x$ and $k$, and give some examples of integers that can be written as
  sums of consecutive powers of primes in more than one way.
  Note that ${\mathcal S}_2(5000)$ was computed by \cite{TWS2022};
    see also sequence A340771 at the On-Line Encyclopedia of
    Integer Sequences (OEIS.org) \cite{OEIS}.
\end{itemize}
We begin by describing our enumeration algorithm in the next section.

\section{The Enumeration Algorithm\label{sec:alg}}

Given as input a bound $x$ and integer exponent $k>1$,
our algorithm produces the elements of the set $\Sk(x)$ as follows.

Let $p_1=2, p_2=3, \ldots $ denote the primes,
  and let $\pi(y)$ denote the number of primes $\le y$.
By the prime number theorem (see, for example, \cite{HW}),
  $\pi(y) \sim y/\log y$, and thus $p_\ell\sim \ell\log \ell$.

We assume all arithmetic operations take constant time.
In practice, all our integers are at most 128 bits, 
  or roughly 38 decimal digits.

\begin{enumerate}
  \item Find the primes up to $x^{1/k}$.

    This step is not the bottleneck, so the Sieve of Eratosthenes is sufficient,
    taking $O( x^{1/k} \log\log x )$ time.
    See also \cite{AB2004,Sorenson2015}.

  \item Compute the prefix sum array $f[\,]$, where 
    $f[0]=0$ and $f[i]:=p_1^k+p_2^k+\cdots+p_i^k$ for all $i\le \pi(x^{1/k})$,
    so that $f[i+1]=f[i]+p_{i+1}^k$.

    Note that the value of the largest entry 
      in the array is bounded by $x^{1+1/k}$.

    Using a binary algorithm for integer exponentiation,
    this takes time $O( \pi( x^{1/k} ) \log k)$, 
    which is smaller than the asymptotic bound given for Step 1.
    Storing $f[\,]$ uses $O(kx^{1/k})$ bits of space.

  \item Loops to enumerate $\Sk(x)$:
    \begin{tabbing}MM\=MM\=MM\=MM\=\kill
      \>\>for $b:=0$ to $\pi(x^{1/k})-1$ do: \+\+\+\\
        for $t:=b+1$ to $\pi(x^{1/k})$ do: \+\\
          $n:=f[t]-f[b]$; \\
          if $n>x$ break the $t$ loop, \\
	  else output($n$, $p_{b+1}$);
    \end{tabbing}

    The time this step takes is proportional to the number of
      ($n$, $p_{b+1}$) pairs that are output,
      which is $\sk(x)$.
  This, in turn, we bound asymptotically in Theorem \ref{thm:main} below,
      at $c_k x^{2/(k+1)} / (\log x)^{2k/(k+1)}$ time.

\end{enumerate}
We output pairs $(n,p_{b+1})$ in case a specific value of $n$ gets repeated.  
If we have repeats for $n$, the $p_{b+1}$ values will be
different, and $p_{b+1}$ is the first prime in the sequence of powers
of primes to generate $n$, allowing us to quickly reconstruct two (or more)
representations of $n$ as $k$th powers of consecutive primes.

In practice, we found repeated values of $n$ to be quite rare.

\subsection*{Example}

Let us compute $\mathcal{S}_3(1000)$ for an example.
\begin{enumerate}
  \item We find the primes up to $1000^{1/3}=10$, so
    $2,3,5,7$.
  \item We compute the prefix array $f[\,]$ as follows:
	\begin{quote}
	  \begin{tabular}{|c|c|c|c|c|} \hline
             0&1&2&3&4 \\
            0&8&35&160&503 \\ \hline
	\end{tabular}
	\end{quote}
  \item We generate the $f[t]-f[b]$ values,
	  and hence $\mathcal{S}_3(1000)$, as follows:
	\begin{quote}
		$b=0:$ $(8,2),(35,2),(160,2),(503,2)$ \\
		$b=1:$ $(27,3),(152,3),(495,3)$ \\
		$b=2:$ $(125,5),(468,5)$ \\
		$b=3:$ $(343,7)$
	\end{quote}
  And thus $s_3(1000)=10$.
\end{enumerate}

\section{A Counting Algorithm\label{sec:cnt}}

In this section we describe an algorithm that computes $\sk(x)$,
  with multiplicity.
Since we are not explicitly constructing all the representations of the 
  integers,
  we are able to save a considerable amount of time
  by collapsing the inner loop in step 3 of the previous algorithm.
\begin{enumerate}
\item Find the primes up to $x^{1/k}$.
\item Compute the prefix sum array $f[\,]$ as done above.
\item Loop to compute the count:
\begin{tabbing}MM\=MM\=\kill
$count:=0$;
$t:=0$; \\
while($t<\pi(x^{1/k})$ and $f[t+1]\le x$) do: \+\\
  $t:=t+1$; \-\\

for $b:=0$ to $\pi(x^{1/k})-1$ do: \+\\
  if $t<\pi(x^{1/k})$ and
   $f[t+1]-f[b]\le x$ then \+\\ $t:=t+1$; \- \\
	$count:=count+(t-b)$;
\end{tabbing}
\end{enumerate}
It is easy to see that the running time for the last step
  is $O(\pi(x^{1/k}))$ arithmetic operations.
So finding the primes in step 1 dominates the running time,
  at $O( x^{1/k}/\log\log x)$ time using, say,
  the Atkin-Bernstein prime sieve \cite{AB2004}.

\section{Analysis\label{sec:anal}}

In this section we prove the following theorem, which provides an
  upper bound on $\sk(x)$.

\begin{thm}\label{thm:main}
  For $k>1$ we have
  $$
    S_k(x)\le(1+o(1))\cdot c_k \frac{ x^{2/(k+1)} }{ (\log x)^{2k/(k+1)} },
  $$
  where $c_k=(k^2/(k-1))\cdot (k+1)^{1-1/k}$.
\end{thm}
Note that $c_k\sim k^2$ for large $k$.

In \cite{TWS2022} they prove the explicit bound
  $$
    s_2(x) \le 28.4201 \frac{ x^{2/3} }{ (\log x)^{4/3} }.
  $$
We also have the trivial lower bound 
  $\sk(x) \ge \pi( x^{1/k} ) \sim k x^{1/k} / \log x$ 
by the prime number theorem.

Our proof follows the same lines as in \cite{TWS2022}.
We begin by partitioning the members of $\Sk(x)$ by the number of
  prime powers $m$ in their representative sum.
Define
$$ 
\skm(x)=\#\{n\le x : \exists \ell\ge 0 : 
  n=p_{\ell+1}^k+\cdots+p_{\ell+m}^k  \}
$$
so that $\sk(x) = \sum_{m=1}^M \skm(x)$ for
  a sufficiently large, and as yet unknown value $M=M(x,k)$,
  the length of the longest sum of powers of consecutive primes $\le x$.

\begin{lemma}[\cite{TWS2022}]
  $$ \skm(x) \le \pi( (x/m) ^{1/k} ). $$
\end{lemma}
\begin{proof}
  Let $\ell=\skm(x)$.  We have
$$
  mp_\ell^k \le p_\ell^k + p_{\ell+1}^k +\cdots + p_{\ell+(m-1)}^k \le x.
$$
  Thus $mp_\ell^k \le x$, or $p_\ell\le (x/m)^{1/k}$, 
    or $\skm(x)=\ell\le \pi( (x/m)^{1/k} )$.
\end{proof}

Next, we need an estimate for $M$.

\begin{lemma}\label{lemma:M}
	$$ M(x,k) \sim (k+1)\frac{x^{1/(k+1)}}{(\log x)^{k/(k+1)}} $$
\end{lemma}

\begin{proof}
We have $$\sum_{\ell=1}^{M} p_\ell^k \le x < \sum_{\ell=1}^{M+1} p_\ell^k .$$
Using the asymptotic estimate $p_M\sim M\log M$ from the prime number theorem
  and using the methods from \cite[\S 2.7]{BS} we have
\begin{eqnarray*}
\sum_{\ell=1}^{M} p_\ell^k 
  & \sim & 
    \sum_{p\le M\log M} p^k 
  \quad=\quad 
    \int_2^{M\log M} t^k d\pi(t) \\
  & \sim & 
    \int_2^{M\log M} \frac{t^k}{\log t} dt 
  \quad \sim \quad 
    \frac{1}{\log M} \int_2^{M\log M} t^k dt , \\
%
\end{eqnarray*}
and so we have 
  $$x\sim \frac{(M\log M)^{k+1}  }{(k+1)\log M} . $$
Taking the logarithm of both sides gives us $(k+1)\log M \sim \log x $.
We then obtain that
	$$ M \sim (k+1)(x\log x)^{1/(k+1)}/\log x . $$
\end{proof}

We are now ready to prove Theorem \ref{thm:main}.

\begin{proof}
We have
$$
\sk(x) = \sum_{m=1}^M \skm(x) 
	\le \sum_{m=1}^M \pi( (x/m) ^{1/k} ).
$$
By the prime number theorem and our lemmas, we have
\begin{eqnarray*}
\sk(x)&\le& \sum_{m=1}^M \pi( (x/m) ^{1/k} ) 
    \quad\sim\quad \sum_{m=1}^M  k (x/m)^{1/k}/\log(x/m)    \\
    &\sim& \frac{k x^{1/k}}{\log x}
        \sum_{m=1}^M  m^{-1/k}  
    \quad\sim\quad \frac{k x^{1/k}}{\log x}
        \frac{ M^{1-1/k} }{1-1/k }  \\
\end{eqnarray*}
Plugging in our estimate for $M$ from Lemma \ref{lemma:M} 
	gives this bound for $\sk(x)$:
   $$ \frac{kx^{1/k}}{\log x}
        \frac{ k}{k-1 } 
   \left((k+1)\frac{x^{1/(k+1)}}{(\log x)^{k/(k+1)}}\right)^{1-1/k}. $$
A bit of algebra simplifies the exponents to complete the proof.
\end{proof} \\
Note that $\lim_{k\rightarrow\infty} (1/k)\cdot (k+1)^{1-1/k} =1$.

We wrap up this section with our lower bound proof for $\sk(x)$:
\begin{thm}\label{thm:lower}
$$
  \sk(x) \ge \binom{M}{2} \ge (1+o(1))\cdot
	\frac{(k+1)^2}{2} \frac{ x^{2/(k+1)} }{ (\log x)^{2k/(k+1)} }.
$$
\end{thm}
\begin{proof}
Take the sum $p_1^k+\cdots+p_M^k$ which is $\le x$ by definition.
The lower bound is obtained by counting the number of $i,j$ pairs with
  $1\le i \le j \le M$, which is $\binom{M}{2}$, as each sum
	$p_i^k+\cdots+p_j^k$ gives an integer $n$ counted by $\sk(x)$.
  Thus, $\sk(x) \ge M(M-1)/2$, and apply Lemma \ref{lemma:M}.
\end{proof}

\section{Empirical Results\label{sec:data}}

In this section we give some of our empirical results.
This is not everything we have -- the interested reader is encouraged
  to contact the second author for copies of the data or source code.

\subsection{Tightness of Theorems \ref{thm:main} and \ref{thm:lower}}

Here we present values of $\sk(x)$ for $k=2,3,5,10,20$ for $x$ up to
  $10^{38}$, which is close to the limit for 128-bit hardware integer
  arithetic.
We also include the upper bound from Theorem \ref{thm:main} and
  the lower bound from Theorem \ref{thm:lower}.

{\small
\begin{quote} \begin{tabular}{l|rrr}
$x$ & $s_{2}(x)$ & Upper & Lower  \\
\hline
$10^{3}$ & 37 & 52 & 34 \\
$10^{4}$ & 132 & 166 & 108 \\
$10^{5}$ & 519 & 574 & 372 \\
$10^{6}$ & 1998 & 2089 & 1357 \\
$10^{7}$ & 7840 & 7898 & 5130 \\
$10^{8}$ & 31372 & 30681 & 19928 \\
$10^{9}$ & 126689 & 121714 & 79056 \\
$10^{10}$ & 517191 & 490907 & 318853 \\
$10^{11}$ & 2132474 & 2006670 & 1303370 \\
$10^{12}$ & 8867094 & 8293885 & 5387036 \\
$10^{13}$ & 37153225 & 34599930 & 22473314 \\
$10^{14}$ & 156713533 & 145488607 & 94497622 \\
$10^{15}$ & 665005737 & 615948906 & 400070550 \\
\end{tabular} \end{quote}
\begin{quote} \begin{tabular}{l|rrr}
$x$ & $s_{3}(x)$ & Upper & Lower  \\
\hline
$10^{3}$ & 10 & 19 & 13 \\
$10^{4}$ & 29 & 40 & 28 \\
$10^{5}$ & 70 & 91 & 64 \\
$10^{6}$ & 186 & 220 & 155 \\
$10^{7}$ & 491 & 554 & 390 \\
$10^{8}$ & 1297 & 1434 & 1011 \\
$10^{9}$ & 3501 & 3801 & 2681 \\
$10^{10}$ & 9568 & 10262 & 7240 \\
$10^{11}$ & 26429 & 28130 & 19846 \\
$10^{12}$ & 73575 & 78071 & 55080 \\
$10^{13}$ & 206617 & 218951 & 154472 \\
$10^{14}$ & 584184 & 619541 & 437093 \\
$10^{15}$ & 1663904 & 1766547 & 1246320 \\
$10^{16}$ & 4769563 & 5070868 & 3577556 \\
$10^{17}$ & 13742399 & 14641613 & 10329827 \\
$10^{18}$ & 39796129 & 42496537 & 29981799 \\
$10^{19}$ & 115807012 & 123917289 & 87425082 \\
$10^{20}$ & 338386013 & 362841801 & 255989092 \\
\end{tabular} \end{quote}
\begin{quote} \begin{tabular}{l|rrr}
$x$ & $s_{5}(x)$ & Upper & Lower  \\
\hline
$10^{5}$ & 10 & 20 & 14 \\
$10^{6}$ & 21 & 32 & 22 \\
$10^{7}$ & 38 & 54 & 37 \\
$10^{8}$ & 68 & 94 & 65 \\
$10^{9}$ & 127 & 167 & 115 \\
$10^{10}$ & 243 & 302 & 208 \\
$10^{11}$ & 479 & 556 & 382 \\
$10^{12}$ & 862 & 1037 & 712 \\
$10^{13}$ & 1639 & 1956 & 1343 \\
$10^{14}$ & 3128 & 3725 & 2558 \\
$10^{15}$ & 6053 & 7154 & 4913 \\
$10^{16}$ & 11799 & 13841 & 9507 \\
$10^{17}$ & 22938 & 26954 & 18513 \\
$10^{18}$ & 44869 & 52794 & 36262 \\
$10^{19}$ & 87959 & 103940 & 71393 \\
$10^{20}$ & 173621 & 205585 & 141209 \\
$10^{21}$ & 343199 & 408328 & 280466 \\
$10^{22}$ & 681611 & 814086 & 559167 \\
$10^{23}$ & 1359330 & 1628652 & 1118664 \\
$10^{24}$ & 2717318 & 3268557 & 2245058 \\
$10^{25}$ & 5451410 & 6578721 & 4518694 \\
$10^{26}$ & 10962586 & 13276572 & 9119214 \\
$10^{27}$ & 22107170 & 26859747 & 18449024 \\
$10^{28}$ & 44656828 & 54464244 & 37409592 \\
$10^{29}$ & 90459929 & 110673813 & 76017986 \\
$10^{30}$ & 183613129 & 225340599 & 154778606 \\
$10^{31}$ & 373421607 & 459662117 & 315725893 \\
$10^{32}$ & 761023562 & 939272425 & 645153503 \\
\end{tabular} \end{quote}
\begin{quote} \begin{tabular}{l|rrr}
$x$ & $s_{10}(x)$ & Upper & Lower  \\
\hline
$10^{10}$ & 10 & 21 & 13 \\
$10^{11}$ & 15 & 26 & 16 \\
$10^{12}$ & 21 & 35 & 22 \\
$10^{13}$ & 36 & 45 & 28 \\
$10^{14}$ & 45 & 61 & 38 \\
$10^{15}$ & 56 & 81 & 51 \\
$10^{16}$ & 78 & 110 & 69 \\
$10^{17}$ & 120 & 150 & 94 \\
$10^{18}$ & 154 & 206 & 129 \\
$10^{19}$ & 214 & 284 & 178 \\
$10^{20}$ & 301 & 393 & 247 \\
$10^{21}$ & 439 & 547 & 344 \\
$10^{22}$ & 599 & 765 & 481 \\
$10^{23}$ & 832 & 1072 & 674 \\
$10^{24}$ & 1187 & 1508 & 949 \\
$10^{25}$ & 1678 & 2129 & 1339 \\
$10^{26}$ & 2373 & 3013 & 1895 \\
$10^{27}$ & 3304 & 4276 & 2690 \\
$10^{28}$ & 4817 & 6083 & 3827 \\
$10^{29}$ & 6786 & 8674 & 5457 \\
$10^{30}$ & 9744 & 12396 & 7799 \\
$10^{31}$ & 13788 & 17751 & 11168 \\
$10^{32}$ & 19871 & 25467 & 16022 \\
$10^{33}$ & 28290 & 36601 & 23027 \\
$10^{34}$ & 40949 & 52692 & 33150 \\
$10^{35}$ & 58459 & 75976 & 47799 \\
$10^{36}$ & 84393 & 109711 & 69023 \\
$10^{37}$ & 121302 & 158647 & 99810 \\
$10^{38}$ & 175797 & 229717 & 144523 \\
\end{tabular} \end{quote}
\begin{quote} \begin{tabular}{l|rrr}
$x$ & $s_{20}(x)$ & Upper & Lower  \\
\hline
$10^{20}$ & 10 & 20 & 12 \\
$10^{21}$ & 15 & 23 & 13 \\
$10^{22}$ & 15 & 26 & 15 \\
$10^{23}$ & 21 & 30 & 17 \\
$10^{24}$ & 21 & 35 & 20 \\
$10^{25}$ & 28 & 40 & 23 \\
$10^{26}$ & 36 & 46 & 27 \\
$10^{27}$ & 36 & 54 & 31 \\
$10^{28}$ & 45 & 63 & 36 \\
$10^{29}$ & 45 & 73 & 42 \\
$10^{30}$ & 66 & 85 & 49 \\
$10^{31}$ & 66 & 100 & 58 \\
$10^{32}$ & 78 & 117 & 68 \\
$10^{33}$ & 105 & 138 & 80 \\
$10^{34}$ & 120 & 162 & 94 \\
$10^{35}$ & 136 & 191 & 111 \\
$10^{36}$ & 171 & 225 & 131 \\
$10^{37}$ & 190 & 266 & 154 \\
$10^{38}$ & 232 & 315 & 183 \\
\end{tabular} \end{quote}
}


Next, we have graphed our upper bound (green) 
  and lower bound (blue) estimates with
  the exact counts (purple) from the tables above to show
  how tight our bounds are in practice.
In our graph for $k=2$ we also include the explicit upper bound (orange)
  from \cite{TWS2022}.

\begin{quote}
  \includegraphics[width=0.90\textwidth]{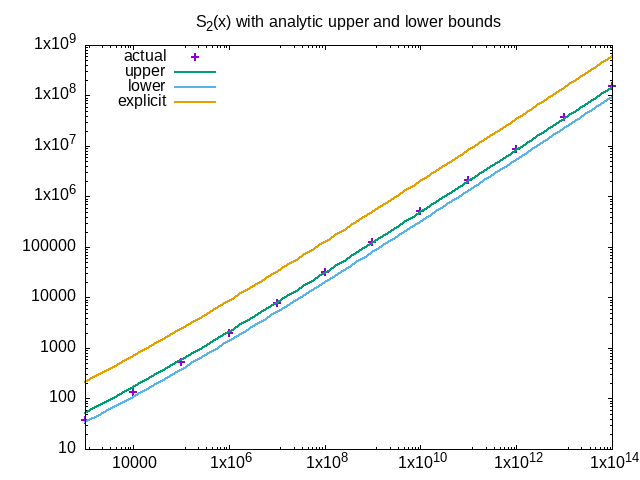} \\
  \includegraphics[width=0.90\textwidth]{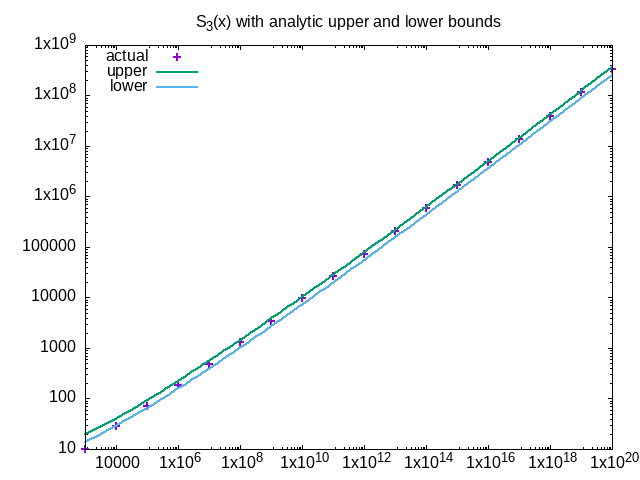} \\
  \includegraphics[width=0.90\textwidth]{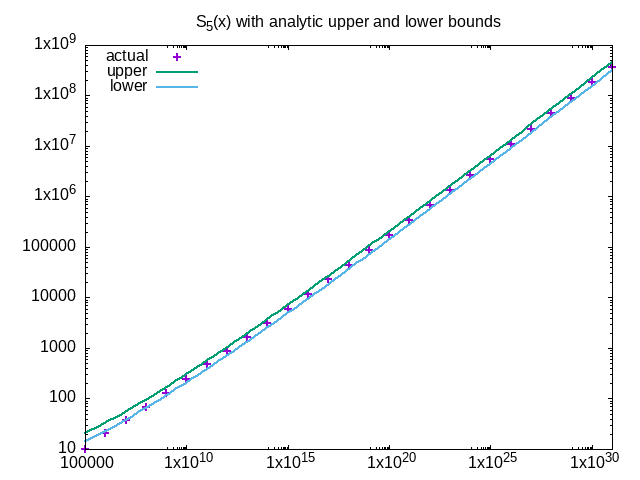} \\
  \includegraphics[width=0.90\textwidth]{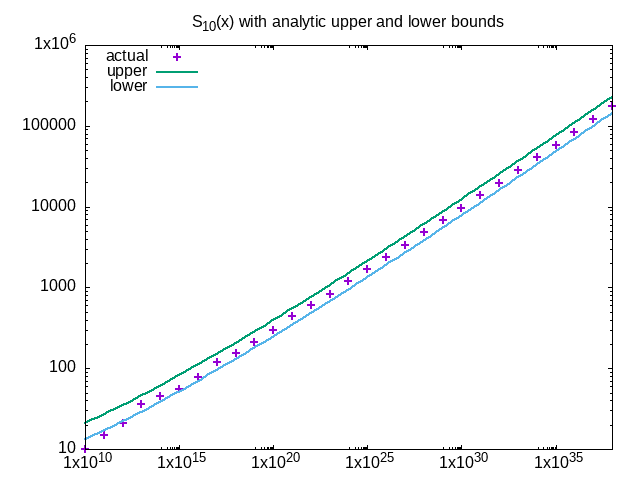} \\
  \includegraphics[width=0.90\textwidth]{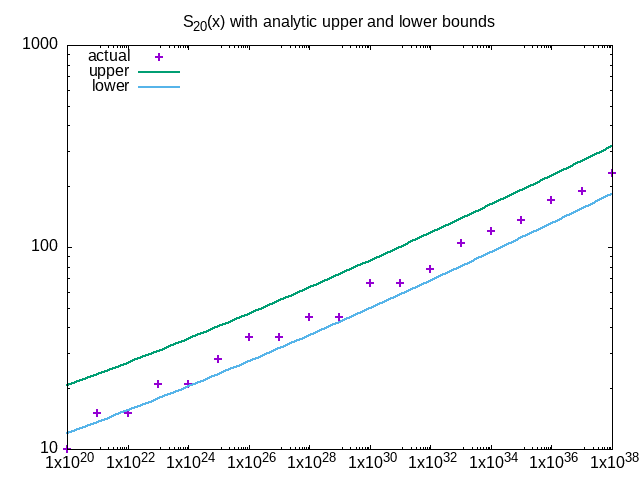}
\end{quote}

We conclude this subsection with a comparison of the constants
  $c_k$ from Theorem \ref{thm:main} with the value of
  $(k+1)^2/2$ from Theorem \ref{thm:lower}.
\begin{quote}
	{\small 
\begin{tabular}{r|r|r}
$k$ & $c_k$ & $(k+1)^2/2$ \\ \hline
2& 6.92  &4.5 \\
3& 11.33 & 8.0 \\
4& 17.83 & 12.5 \\
5& 26.21  &18.0 \\
6& 36.44 & 24.5 \\
7& 48.54 & 32.0 \\
8& 62.52 & 40.5 \\
9& 78.39 & 50.0 \\
10& 96.16&  60.5 \\
11& 115.84 & 72.0 \\
12& 137.43&  84.5 \\
13& 160.94&  98.0 \\
14& 186.38&  112.5 \\
15& 213.75&  128.0 \\
16& 243.05&  144.5 \\
17& 274.29&  162.0 \\
18& 307.47&  180.5 \\
19& 342.60&  200.0 \\
20& 379.68&  220.5 \\
\end{tabular}
	}
\end{quote}

\subsection{Duplicates}

We found 40 values of $n\le x=10^{12}$ that have multiple representations
  as sums of consecutive squares of primes.
The smallest such number is $14720439$, which can be written as
\begin{quote}
$941^2+947^2+953^2+967^2+971^2+977^2+983^2+991^2+997^2+1009^2+
1013^2+1019^2+1021^2+1031^2+1033^2$
\end{quote}
and as
\begin{quote}$
131^2+137^2+139^2+149^2+151^2+157^2+163^2+167^2+173^2+179^2+181^2+191^2+193^2+197^2+199^2+211^2+223^2+227^2+229^2+233^2+239^2+241^2+251^2+257^2+263^2+269^2+271^2+277^2+281^2+283^2+293^2+307^2+311^2+313^2+317^2+331^2+337^2+347^2+349^2+353^2+359^2+367^2+373^2+379^2+383^2+389^2+397^2+401^2+409^2+419^2+421^2+431^2+433^2+439^2+443^2+449^2+457^2+461^2+463^2+467^2+479^2+487^2+491^2+499^2+503^2+509^2+521^2+523^2+541^2+547^2+557^2+563^2+569^2+571^2+577^2+587^2+593^2+599^2+601^2+607^2+613^2+617^2+619^2+631^2+641^2+643^2+647^2.
$
\end{quote}
To find these, we sorted the output of our algorithm from \S2,
  and then used the \texttt{uniq -D} unix/linux command to suss out
  the duplicates.

We found no integers that can be written as the sum
  of consecutive powers of primes in more than one way
  for any power larger than 2.
We searched for cubes up to $10^{18}$, fifth powers up to $10^{27}$,
  and tenth and twentieth powers up to $10^{38}$.
This search requires computing $\Sk(x)$ and not just $\sk(x)$;
note that it is much faster to compute just $\sk(x)$ in practice
  because outputting the elements of $\Sk(x)$ to a text file
  slows down the computation considerably.

We found exactly one example with differing powers:
\begin{eqnarray*}
23939 &=& 23^2+29^2+31^2+37^2+41^2+43^2+47^2+53^2+59^2+61^2+67^2 \\
      &=& 17^3+19^3+23^3.
\end{eqnarray*}

We conclude this subsection with the list of 40 integers $\le 10^{12}$
  that can be written as sums of squares of consecutive primes in
  two ways.
For each such integer in the table below, we list the starting primes
  for each of their two ways to sum.
  \begin{quote}\footnotesize \begin{tabular}{l|ll}
	  $n$ & Prime 1 & Prime 2 \\ \hline
14720439 & 131 & 941 \\
16535628 & 1123 & 569 \\
34714710 & 2389 & 401 \\
40741208 & 131 & 653 \\
61436388 & 569 & 809 \\
603346308 & 401 & 919 \\
1172360113 & 3701 & 4673 \\
1368156941 & 1367 & 16519 \\
1574100889 & 3623 & 613 \\
1924496102 & 11657 & 2803 \\
1989253499 & 3359 & 613 \\
2021860243 & 3701 & 4297 \\
6774546339 & 11273 & 47513 \\
9770541610 & 1663 & 7243 \\
12230855963 & 10177 & 2777 \\
12311606487 & 28603 & 3257 \\
12540842446 & 11087 & 479 \\
14513723777 & 1663 & 6323 \\
26423329489 & 1709 & 32401 \\
38648724198 & 2777 & 6967 \\
  \end{tabular}
  \end{quote} 

  \begin{quote}\footnotesize \begin{tabular}{l|ll}
	  $n$ & Prime 1 & Prime 2 \\ \hline
47638558043 & 28097 & 65731 \\
50195886916 & 479 & 6857 \\
50811319931 & 2039 & 21283 \\
56449248367 & 2803 & 4127 \\
86659250142 & 4561 & 53609 \\
105146546059 & 29587 & 6599 \\
119789313426 & 31847 & 42299 \\
125958414196 & 16763 & 26183 \\
134051910100 & 183047 & 4397 \\
159625748030 & 1367 & 3301 \\
169046403821 & 183829 & 19717 \\
263787548443 & 47297 & 62347 \\
330881994258 & 11161 & 2039 \\
438882621700 & 16763 & 20369 \\
507397251905 & 643 & 75013 \\
572522061248 & 18427 & 44371 \\
687481319598 & 16139 & 338461 \\
780455791261 & 3257 & 7057 \\
847632329089 & 184003 & 7523 \\
854350226239 & 14821 & 6599 
  \end{tabular}
  \end{quote} 

\subsection{Initial Elements of $\Sk$}

We wrap up the presentation of our computations with the first few
  elements of each of the $\Sk$ sets we computed.

$\mathcal{S}_2$: \\
4 9 13 25 34 38 49 74 83 87 121 169 170 195 204 208 289 290 339 361 

$\mathcal{S}_3$: \\
8 27 35 125 152 160 343 468 495 503 1331 1674 1799 1826 1834 2197 3528 3871  3996 4023

$\mathcal{S}_5$: \\
32 243 275 3125 3368 3400  16807 19932 20175 20207 161051  177858 180983 181226  181258 371293 532344  549151 552276 552519 

$\mathcal{S}_{10}$: \\
1024 59049 60073 9765625 9824674       9825698      282475249   292240874  292299923 292300947

$\mathcal{S}_{20}$: \\
1048576 3486784401 3487832977 95367431640625 95370918425026 95370919473602 79792266297612001 79887633729252626 79887637216037027 79887637217085603

\section{Future Work\label{sec:future}}

We have several ideas for future work:
\begin{itemize}
  \item
Our primary goal is to parallelize our algorithm from \S2 to extend our
  computations.
For larger powers, this will also mean using multiple-precision integer
  arithmetic using, for example, GMP.
  \item
A more careful proof of Theorem \ref{thm:main} might give explicit upper
  bounds, or perhaps an asymptotic constant.
  If such a constant exists, it appears to be near 0.6.
  \item
Is there a power $k>2$ for which there are integers with multiple
  representations as sums of powers of consecutive primes?
We have not found any as of yet.
    
\end{itemize}

\section*{Acknowledgements}

The authors are grateful to Frank Levinson for his support of
  computing research infrastructure at Butler University.

\bibliographystyle{plain}

\end{document}